\documentclass[]{amsart}
\pdfoutput=1

\usepackage[utf8]{inputenc}
\usepackage[T1]{fontenc}
\usepackage{hyperref}
\usepackage{amsfonts,amssymb,amsthm,amsmath}
\usepackage{algorithm}
\usepackage{algpseudocode}
\usepackage{graphbox,graphicx}
\usepackage{url}
\usepackage{mleftright}
\usepackage{bbm}
\usepackage{comment}
\usepackage{xcolor}

\usepackage[backend=biber,maxbibnames=9]{biblatex}
\calclayout

\graphicspath{{figs/}}

\newtheorem{theorem}{Theorem}
\newtheorem{proposition}[theorem]{Proposition}
\newtheorem{assumption}{Assumption}
\newtheorem{example}[theorem]{Example}
\newtheorem{lemma}[theorem]{Lemma}
\newtheorem{definition}[theorem]{Definition}
\newtheorem{corollary}[theorem]{Corollary}

\newtheorem{remark}[theorem]{Remark}
\numberwithin{theorem}{section}

\DeclareMathOperator*{\argmin}{arg\,min}

\makeatletter
\@namedef{subjclassname@2020}{%
	\textup{2020} Mathematics Subject Classification}
\makeatother

\title[Heat kernel asymptotics for scaling limits of isoradial graphs]{Heat kernel asymptotics for scaling limits of \\ isoradial graphs}

\author{Simon Schwarz}
\address{Institute for Mathematical Stochastics\\ 
	University of Göttingen\\
	Goldschmidtstraße 7\\ 37077 Göttingen\\
	Germany}
\email{simon.schwarz@uni-goettingen.de}

\author{Anja Sturm}
\address{Institute for Mathematical Stochastics\\ 
	University of Göttingen\\
	Goldschmidtstraße 7\\ 37077 Göttingen\\
	Germany}
\email{anja.sturm@mathematik.uni-goettingen.de}

\author{Max Wardetzky}
\address{Institute for Numerical and Applied Mathematics\\
	University of Göttingen \\
	Lotzestr. 16-18 \\ 37083 Göttingen\\
	Germany}
\email{wardetzky@math.uni-goettingen.de}

\thanks{The authors acknowledge financial support by the Deutsche Forschungsgemeinschaft (DFG, German Research Foundation) via project A02 of the Collaborative Research Center 1456. We sincerely thank an anonymous reviewer for the very constructive comments.}

\begin{document}
	
	
	
	\begin{abstract}
		We consider the asymptotics of the discrete heat kernel on isoradial graphs for the case where the time and the edge lengths tend to zero simultaneously. Depending on the asymptotic ratio between time and edge lengths, we show that two different regimes arise: (i)~a Gaussian regime and (ii) a Poissonian regime, which resemble the short-time asymptotics of the heat kernel on (i) Euclidean spaces and (ii) graphs, respectively.
	\end{abstract}
	\vspace*{-1cm}
	\maketitle
	
	\section{Introduction}\label{sec:intro}
	The short-time asymptotics of the heat kernel on continuous and discrete domains differ substantially: 
	A classical result by Varadhan \cite{varadhan1967} shows that for the heat kernel $q$ on a Riemannian manifold $(M,g)$
	\begin{equation}\label{eq:cont-asymp1}
	\lim_{t\rightarrow 0} -2t\log q_t (x,y) = d_g^2(x,y)
	\end{equation}
	holds uniformly for all $x,y\in M$.
	On the other hand, for a weighted graph $(V,E)$ equipped with its graph Laplacian,
	the discrete heat kernel $p$ satisfies
	\begin{equation}\label{eq:cont-asymp2}
	\lim_{t\rightarrow 0} \frac{1}{\log t}\log p_t (x,y) = d_c(x,y) \ ,
	\end{equation}
	where $d_c$ is the combinatorial graph distance, see \cite{keller2016,steinerberger2018}.
	We show that both behaviors can occur for scaling limits of isoradial graphs if the time decreases simultaneously with the edge lengths of the graph.
	
	More precisely, let $(\Gamma_h)_{h>0}$ be a family of planar graphs embedded isoradially in $\mathbb{C}$, i.e., each face of $\Gamma_h$ is inscribed into a circle of diameter $h$ and the midpoints of the circles lie inside their corresponding face. Suppose that this family satisfies mild assumptions on the minimal edge lengths, cf. Assumptions \ref{ass:bap-primal} and \ref{ass:bap-dual}. Furthermore, let $d_h$ be the combinatorial graph distance on $\Gamma_h$, $p^h$ the discrete heat kernel defined with respect to the geometric graph Laplacian on isoradial graphs, see Definition \ref{def:laplacian} in Section \ref{sec:prelim}, and $\tau > 0$ a constant. 
	Our main Theorems \ref{thm:euclideanregime} and \ref{thm:graphregime} show that for $\beta < 1$ 
	\begin{equation}\label{eq:graph-asymp1}
	\lim_{h\rightarrow 0} -2 \left(h^\beta \tau\right) \log p^h_{h^\beta \tau} (x,y) = 
	\lvert x-y\rvert^2
	\end{equation}
	holds, while for $\beta > 1$ we obtain 
	\begin{equation}\label{eq:graph-asymp2}
	\frac{h}{\log h^{\beta-1}} \log p^{h}_{h^\beta \tau} (x,y) = hd_h (x,y) + o(1)
	\end{equation}
	as $h\to 0$. Notice that the limits in Equation \eqref{eq:graph-asymp1} and \eqref{eq:graph-asymp2} correspond to the short-time asymptotics in Equation \eqref{eq:cont-asymp1} and \eqref{eq:cont-asymp2}, respectively: While the edge lengths decrease faster than time in the case of $\beta < 1$ (leading to the Euclidean regime), the reverse situation for $\beta > 1$ yields the graph behavior. We do not treat the critical case $\beta = 1$ here, because it strongly depends on the underlying graphs. 
	
	The main tool to prove the Euclidean regime is a \emph{large deviation principle}, which is obtained by the fact that the geometric Laplacian 
	maps quadratic functions to constants on isoradial graphs, see Lemma~\ref{lem:second-order-approx}.
	In contrast, the proof for the case $\beta > 1$ does not require any properties of isoradial graphs and also holds in more general settings. 
	
	The existence of two different regimes similar to ours was previously observed in the physics literature for \emph{one-dimensional} continuous-time random walks \cite{barkai2020,wang2020}.
	
	Results on the short-time asymptotic of heat kernels are classical in potential analysis and available for numerous different settings, e.g. for Lipschitz manifolds \cite{norris1997}, fractals \cite{arous2000,kumagai1997} and infinite dimensional spaces \cite{zhang2000}. Heat kernels on graphs are studied in the textbooks \cite{barlow2017,keller2021,telcs2006}. In particular, we use Gaussian heat kernel estimates in the proof of Theorem \ref{thm:euclideanregime}, which have been established under various conditions in, e.g., \cite{delmotte1999,grigoryan2012,hebisch1993}. The heat kernel on a weighted graph is furthermore related to its spectrum \cite{guivarch1980,kotani1998} and its curvature~\cite{bauer2015,horn2019}.
	
	Isoradial graphs were introduced by Duffin in \cite{duffin1968} in order to develop discrete complex analysis and were later used for the rigorous study of models from statistcal physics and their scaling limits at criticality, see, e.g.,~\cite{boutillier2012,chelkak2012,duminil2018}. 
	It has been shown that many concepts of discrete complex analysis on isoradial graphs are consistent with their corresponding continuous versions, see e.g.~\cite{chelkak2011,kenyon2002}.
	
	\section{Preliminaries}\label{sec:prelim}
	\subsection{Isoradial graphs}
	Let $(\Gamma_h)_{h>0}$ be a family of infinite planar graphs embedded into $\mathbb{C}$ with straight edges. For $h>0$, let $V_h$, $E_h$, and $F_h$ denote the vertex set, the edge set, and the set of faces of $\Gamma_h$, respectively. We write $v\sim u$ to indicate that $\{v,u\}\in E_h$. Throughout this article, we consider isoradial graphs, i.e., each face $f\in F_h$ is inscribed into a circle of diameter $h$ and the midpoints of the circles lie inside their corresponding face, see Figure \ref{fig:isoradial} \emph{(left)}. 
	Under the widely used Assumptions \ref{ass:bap-primal} and \ref{ass:bap-dual}, see below, all isoradial graphs have bounded degree.
	
	The isoradial graph $\Gamma_h$ has an isoradial dual graph $\Gamma^*_h$ embedded into $\mathbb{C}$ with vertex set $V_h^* = \{m_f : f\in F_h\}$, where $m_f$ is the center of the circumcircle of the face $f\in F_h$. Two vertices $m_f,m_g\in V^*_h$ are connected if and only if the corresponding faces $f,g$ are adjacent.
	Notice that there is a one-to-one correspondence between faces of $\Gamma_h$ and vertices of $\Gamma^*_h$ and vice-versa.
	Moreover, $\Gamma^*_h$ is a plane graph, and for any edge $e\in E_h$, there exists a unique dual edge $e^* \in E^*_h$ orthogonal to $e$ connecting the circumcenters of the two faces with common edge $e$, see Figure \ref{fig:isoradial} \emph{(middle)}. Every isoradial graph induces a rhombic tiling of the Euclidean plane since the faces of the bipartite graph between the vertex sets $V_h$ and $V_h^*$ are rhombi with edge length $\frac{h}{2}$, see Figure~\ref{fig:isoradial}~\emph{(right)}.
	
	\begin{figure}
		\centering
		\includegraphics[width=\textwidth]{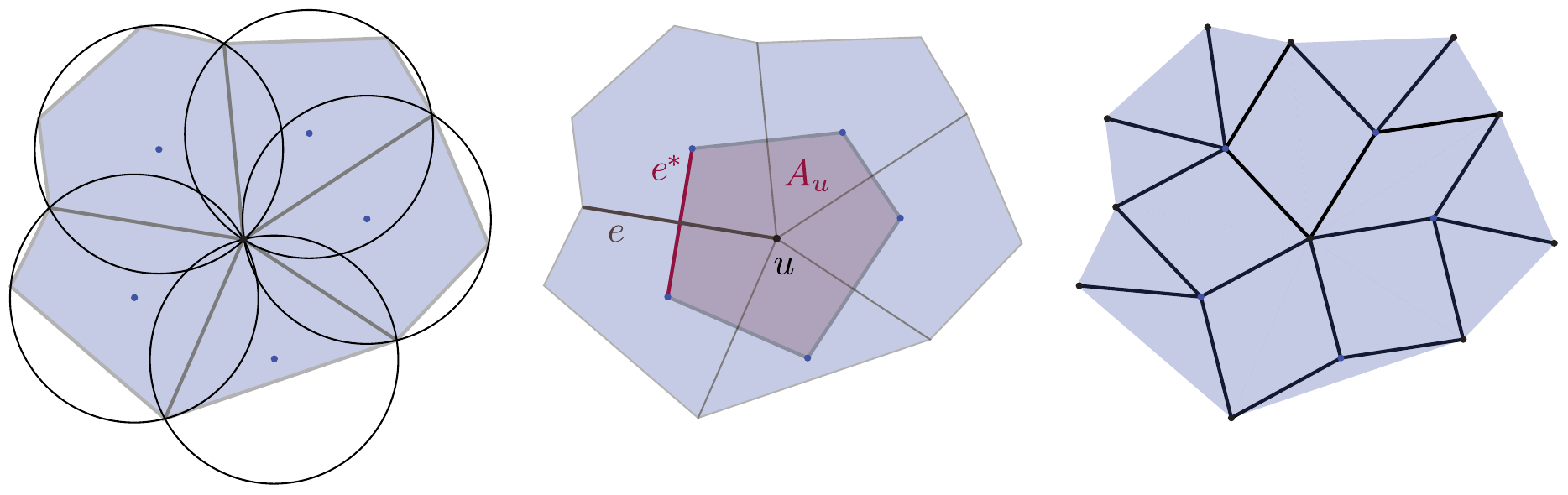}
		\caption{Vertex fan in an isoradial graph. \emph{Left:} All faces are inscribed into circles of equal diameter. \emph{Middle:} 
		Dual edge $e^*$ defined as the segment between midpoints of circumcircles of two adjacent polygons. The resulting polygon enclosed by dual edges is the dual face of $u$ with area~$A_u$. \emph{Right:} Resulting rhombic tiling of the plane.}
		\label{fig:isoradial}
	\end{figure}
	
	We frequently assume that $(\Gamma_h)_{h>0}$ satisfies the following properties:
	
	\begin{assumption}\label{ass:bap-primal}
		There exists a constant $c_p > 0$ such that for all $h>0$
		\begin{equation*}
		c_p h \leq \inf_{e\in E_h} \lvert e \rvert \ ,
		\end{equation*}
		where $\lvert \cdot \rvert$ denotes Euclidean length. 
	\end{assumption}
	
	\begin{assumption}\label{ass:bap-dual}
		There exists a constant $c_d > 0$ such that for all $h > 0$
		\begin{equation*}
		c_d h \leq \inf_{e^*\in E_h^*} \lvert e^*\rvert \ .
		\end{equation*}
	\end{assumption}
	
	Notice that Assumptions \ref{ass:bap-primal} and \ref{ass:bap-dual} combined are equivalent to the \emph{bounded-angle property} with constants independent of $h$, which is often used in the literature on isoradial graphs, see, e.g. \cite{grimmett2013}. The bounded-angle property is satisfied if the induced rhombi do not degenerate. Notice furthermore that the isoradiality of $\Gamma_h$ and $\Gamma^*_h$ already implies that
	\begin{equation}
	\label{edgelengthbound}
	\sup_{e\in E_h} \lvert e \rvert \leq h \; \text{and} \; \sup_{e^*\in E^*_h} \lvert e^* \rvert \leq h \ .
	\end{equation}
	Under Assumption \ref{ass:bap-primal} and \ref{ass:bap-dual}, the maximal degree of $\Gamma_h$ is bounded uniformly in $h$. In the following, we denote this bound by $M$.
	
	For $u,v\in V_h$, consider the set $\mathfrak{X}_h(u,v)$ of paths $\gamma = (q_0,\dots, q_{n(\gamma)})$ in $\Gamma_h$ with $q_0 = u$, $q_{n(\gamma)} = v$ and $q_i \sim q_{i+1}$ for $0\leq i < n(\gamma)$. We write $\{z,z'\}\in\gamma$, if there is an $i<n(\gamma )$ with $q_i = z$ and $q_{i+1} = z'$ and also $z\in\gamma$ if there is an $i\leq n(\gamma )$ with $q_i = z$.
	By using these paths we define two metrics on the graphs: Let
	\begin{equation*}
	l_h (u,v) = \inf_{\gamma\in\mathfrak{X}_h(u,v)}\sum_{e\in \gamma}\, \lvert e\rvert
	\end{equation*}
	denote the \emph{weighted shortest-path distance}, where the sum is taken over all edges connecting the vertices in the respective path. The \emph{combinatorial graph distance} is given by
	\begin{equation*}
	d_h (u,v) = \inf_{\gamma\in\mathfrak{X}_h(u,v)} n (\gamma) \ .
	\end{equation*}
	The next lemma shows that Assumption \ref{ass:bap-primal} implies that both the weighted shortest-path distance $l_h$ and the \emph{rescaled} combinatorial distance $h d_h$ are equivalent to the Euclidean metric:
	
	\begin{lemma}\label{lem:spanning-tree}
		There is a constant $\kappa \leq 1.998$ such that for all $h>0$ and $u,v\in V_h$
		\begin{equation}\label{eq:geometric-spanner}
		\lvert u-v\rvert \leq l_h(u,v) \quad \text{and} \quad l_h (u,v) \leq \kappa \lvert u-v\rvert \ .
		\end{equation}
		If additionally Assumption \ref{ass:bap-primal} holds, then
		\begin{equation}\label{eq:geometric-spanner^c}
		\lvert u-v\rvert \leq h d_h (u,v) \quad \text{and}\quad h d_h (u,v) \leq \frac{\kappa}{c_p} \lvert u-v\rvert
		\end{equation}
		for $h>0$ and $u,v\in V_h$, where $c_p$ is the constant in Assumption \ref{ass:bap-primal}.
	\end{lemma}

	\begin{proof}
		The upper-bound on $l_h$ in \eqref{eq:geometric-spanner} is the only non-trivial part of this lemma. We can make use of Xia's main theorem in \cite{xia2013} to show this inequality:
		
		Fix $u,v\in V_h$. Let $L$ be the line segment from $u$ to $v$ and let $$\mathcal{F}=\{f\in F_h : f\cap L \neq \emptyset\}$$ be the subset of faces crossed by $L$. 
		The circumdisks of the faces in $\mathcal{F}$ induce a chain $(D_1, \dots, D_n)$ such that consecutive disks intersect, no disk contains any other disk, and three consecutive disks intersect in at most a single point. 
		In this situation, Theorem 1 in \cite{xia2013} yields  a constant $\kappa < 1.998$ such that
		\begin{equation*}
		l_h(u,v) \leq \kappa \lvert u-v\rvert \ ,
		\end{equation*}
		which completes the proof. \qedhere
	\end{proof}
	
	
	\subsection{Geometric Graph Laplacian}
	We define a Laplace operator on $\Gamma_h$, $h>0$, by considering the edge weights
	\begin{equation*}
	e = \{u,v\} \mapsto \omega_{uv} := \frac{\lvert e^*\rvert}{\lvert e\rvert}\ .
	\end{equation*}
	These weights are well-defined and non-negative. For later use we also define vertex weights by 
	\begin{equation}\label{eq:alternative-weights}
	m_u := \sum_{z\sim u} \omega_{uz}.
	\end{equation}
	\begin{definition}\label{def:laplacian}
		For any function $f: V_h\rightarrow\mathbb{C}$ defined on the vertices of $\Gamma_h$, we define the geometric graph Laplacian of $f$ as
		\begin{equation}\label{eq:laplacian}
		(\Delta_h f)(u) := \frac{1}{A_u}\sum_{v\sim u} \omega_{uv} (f(v) - f(u)) \ ,
		\end{equation}
		where $A_u$ is the area of the dual face of $u$, see Figure \ref{fig:isoradial} \emph{(middle)}.
	\end{definition}
	
	\begin{remark}
		Notice that $\Delta_h$ is negative semi-definite by our sign convention.
		Our notion of the Laplacian is identical to the one defined in \cite{chelkak2011,kenyon2002}. Moreover, if $\Gamma_h$ is a triangulation, the above Laplacian is (up to the area factor and a sign) identical to the so-called cotan operator, see \cite{macneal1949,pinkall1993}.
	\end{remark}
	
	The operator $\Delta_h$ is \emph{geometric}. 
	In particular, $\Delta_h$ maps affine functions to zero and quadratic functions to constants. More precisely it was proven in \cite{chelkak2011}, Lemma~2.2 that: 
	
	\begin{lemma}\label{lem:second-order-approx}
		Let $a,b,c \in \mathbb{C}$ be constants. Affine functions of the form $f(x+iy) = ax + by +c$ restricted to $\Gamma_h$ are in the kernel of $\Delta_h$. Moreover, for all functions $f(x+iy) = ax^2 + bxy +cy^2$ one has $\Delta_h f = 2(a+c)$.
	\end{lemma}
	
	In later sections, we will frequently use the following bounds on the edge weights 
	and dual areas that follow immediately from Assumptions \ref{ass:bap-primal} and \ref{ass:bap-dual}
	as well as from \eqref{edgelengthbound} and the uniform bound on the maximum degree.
	
	\begin{lemma}\label{lem:bounded-weights}
		Under Assumptions \ref{ass:bap-primal} and \ref{ass:bap-dual}, the edge weights are bounded by 
		\begin{equation*}
		c_d \leq \inf_{e\in E_h} \omega_{e} \leq \sup_{e\in E_h} \omega_{e} \leq c_p^{-1}\ ,
		\end{equation*}
		where the constants $c_p$ and $c_d$ are the same as in the assumptions and independent of $h$. 
		Moreover, there are constants $\kappa_1, \kappa_2 > 0$ independent of $h$ such that
		\begin{equation*}
		\kappa_1 h^{2} \leq \inf_{u\in V_h} A_u \leq \sup_{u\in V_h} A_u \leq \kappa_2 h^2\ .
		\end{equation*}
	\end{lemma}
	
	\subsection{Discrete Heat Kernels}\label{ssec:discrete-hk}
	Let $X^h = (X^h_t)_{t\geq 0}$ be the time-continuous random walk on $\Gamma_h$ with infinitesimal generator $\frac12 \Delta_h$. The process stays in a vertex $u\in V_h$ for a time that is exponentially distributed with rate $$\lambda (u) := 
	\frac{m_u}{2A_u}$$ 
	and moves to an adjacent vertex $v\sim u$ with probability $m_u^{-1} \omega_{uv}$.
	Define the 
	measure $\mathcal{\nu}$ on $V_h$ by $\nu(\{u\}) = A_u$ for any $u\in V_h$. Then, $\nu$ is a stationary measure of $X^h$ since $\Delta_h$ is symmetric with respect to $L^2(V_h,\nu)$.
	
	The \emph{discrete heat kernel} $p^h = (p^h_{t})_{t\geq 0}$ on $\Gamma_h$ is the transition probability density of $X^h$ with respect to $\nu$, i.e.,
		\begin{equation*}
		p_t^h (u,v) = \frac{\mathbb{P}_u[X_t^h = v]}{A_v}
		\end{equation*}
		for any $t\geq 0$ and $u,v\in V_h$, where $\mathbb{P}_u$ denotes the law of $X^h$ with initial condition $X_0^h = u$. It solves the discrete heat equation
	\begin{equation*}\label{eq:heat-equation}
	\begin{split}
	\partial_t p_t^h (u,\cdot ) &= \frac{1}{2}\Delta_h p_t^h (u,\cdot ) \ , \quad t>0\text{, } \\
	p_0^h (u,v) &= \delta_u (v)
	\end{split}
	\end{equation*}
	for all $u,v\in V_h$, where $\delta_u (v) = A_u^{-1}$ if $u=v$ and $0$ otherwise. 
	Processes like $X^h$ with vertex dependent waiting times are also known as \emph{variable speed random walks} in contrast to \emph{constant speed random walks}, see, e.g. \cite{barlow2010}. 
	In the following, we may extend $p^h_t$ implicitly to $\mathbb{C}\times\mathbb{C}$ by $$p^h_t (x,y) := p^h_t (\pi_{V_h} (x), \pi_{V_h} (y))\quad \text{for all } x,y\in\mathbb{C}\ ,$$ 
	where $\pi_{V_h}$ denotes the projection to the closest vertex with respect to the Euclidean distance.
	Likewise, we write $\mathbb{P}_x := \mathbb{P}_{\pi_{V_h} (x)}$ for any $x\in\mathbb{C}$.
	If the nearest vertex is not unique, we choose to map to one of them, in order to make $\pi_{V_h}$ well defined.
	
	\section{Euclidean Limit}
	
	\begin{theorem}\label{thm:euclideanregime}
		Let $(\Gamma_h)_{h>0}$ be a family of isoradial graphs that satisfies Assumptions \ref{ass:bap-primal} and \ref{ass:bap-dual}. For all $\beta \in (0,1)$ and $\tau >0$,
		\begin{equation*}
		\lim_{h\rightarrow 0} - 2 \left(h^\beta \tau \right) \log p_{h^\beta \tau}^h (x,y) = \lvert x-y\rvert^2
		\end{equation*}
		uniformly on compact sets in $\mathbb{C}\times \mathbb{C}$.
	\end{theorem}
	
	Let $X^h=(X^h_t)_{t\geq 0}$ denote the continuous-time random walk with generator $\frac{1}{2} \Delta_h$ from the previous section. 
	In order to prove Theorem \ref{thm:euclideanregime}, we first quantify the probabilities $\mathbb{P}[X^h_{h^\beta t} \in A]$ for measurable sets $A\subset\mathbb{C}$ by proving a \emph{large deviation principle} (LDP) in Section~\ref{ssec:ldp}. Using this LDP and a Gaussian lower bound for the heat kernel, we conclude the proof of Theorem~\ref{thm:euclideanregime} in Section \ref{ssec:densities}.
	
	\subsection{Large deviation principle}\label{ssec:ldp}
	
	Without an additional scaling factor of $h^{-\frac{\beta}{2}}$, the probability $\mathbb{P}_{x}[X^h_{h^\beta t} \in A]$ converges to $0$ as $h\rightarrow 0$ for any measurable $A\subset \mathbb{C}$, $x\in\mathbb{C}$, $\beta > 0$ and fixed $t>0$ if the point $x$ has positive distance to~$A$.  
	Large deviations theory asserts that the respective convergence is exponentially fast, and provides the rate of this exponential decay:
	A family of Borel probability measures $(\mu_\varepsilon)_{\varepsilon > 0}$ on a metric space $(E,d)$ satisfies a large deviation principle (LDP) with rate function $I: E\rightarrow [0,\infty ]$ and speed $(a_\varepsilon)_{\varepsilon > 0}$ if $I$ is lower semi-continuous and not equal to infinity everywhere, $a_\varepsilon\to 0$ as $\varepsilon\to 0$, and
	\begin{equation}\label{eq:ldp-bounds}
	-\inf_{y\in \mathring{A}} I(y) \leq \liminf_{\varepsilon\rightarrow 0} a_\varepsilon \log \mu_\varepsilon (A) \leq \limsup_{\varepsilon\rightarrow 0} a_\varepsilon \log \mu_\varepsilon (A) \leq -\inf_{y\in\overline{A}} I(y)
	\end{equation}
	for all measurable sets $A\subset E$.
	In particular, for $I$-continuous sets, i.e., sets that satisfy $\inf_{y\in \mathring{A}} I(y) = \inf_{y\in\overline{A}} I(y)$, we have
	\begin{equation*}
	\lim_{\varepsilon\rightarrow 0} a_\varepsilon \log \mu_\varepsilon (A) = -\inf_{y\in A} I(y) \ ,
	\end{equation*}
	which is reminiscent of the asymptotics in Theorem \ref{thm:euclideanregime} when expressed in terms of densities. 
	For the Wiener measure and path space measures of simple random walks, large deviation principles are given by Schilder's and Mogulskii's theorems, respectively, see, e.g.~\cite{dembo2010}. Since our setting deviates from simple random walks, we are concerned with proving an LDP for the measures induced by $\left(X^h_{h^\beta \cdot}\right)_{h>0}$ on the space of càdlàg functions, i.e., right-continuous with left limits, equipped with the locally uniform topology.
	
	We introduce some additional notation: For any square integrable martingales $Z$ and $\widetilde{Z}$ with respect to a filtration $(\mathcal{F}_t)_{t\geq 0}$ we denote the \emph{predictable} quadratic covariation by $\langle Z, \widetilde{Z}\rangle$, i.e., the unique predictable and increasing process such that $Z\widetilde{Z} - \langle Z, \widetilde{Z}\rangle$ is a local martingale wrt.\ $(\mathcal{F}_t)_{t\geq 0}$. The predictable quadratic variation of $Z$ is then given as $\langle Z \rangle := \langle Z, Z\rangle$. Existence and uniqueness of the predictable quadratic variation can be proven by using the Doob--Meyer decomposition theorem.
	Furthermore, we write $x_{t-} := \lim_{s\uparrow t} x_s$ 
	for $t>0$ and a càdlàg path $x: [0,\infty) \rightarrow \mathbb{C}$. By convention, $x_{0-} = x_0$.
	Finally, with $X^h$ defined as above, let
	$$(M_t^{\beta, h})_{t\geq 0} := \left(h^{-\frac{\beta}{2}} X_{h^\beta t}^h\right)_{t\geq 0}\quad \text{with } M_0^{\beta, h} = \pi_{V_h} (0)\ ,$$
	where $\pi_{V_h}$ denotes the projection to the closest vertex, see Section \ref{ssec:discrete-hk}.
	Denote the real and imaginary parts of this process by $M^{\beta, h, (1)}$ and $M^{\beta, h, (2)}$, respectively.
	
	\begin{lemma}\label{lem:quad-var}
		For $\beta \geq 0$, $(M^{\beta, h}_t)_{t\geq 0}$ is a martingale with respect to the natural filtration,
		\begin{equation*}
		\langle M^{\beta, h, (j)}\rangle_t = t \quad \text{and} \quad \langle M^{\beta, h, (1)},M^{\beta, h, (2)} \rangle_t = 0
		\end{equation*}
		for $j\in \{1,2\}$ and $h>0$.
	\end{lemma}
	
	\begin{proof}
		For any vertex $u\in V_h$, let $T_u$ be an exponential distributed random variable with rate
		\begin{equation*}
		\lambda (u) = \frac{1}{2A_u} \sum_{z\sim u} \omega_{uz}\ ,
		\end{equation*}
		describing the time $M^{\beta,h}$ spends in $u$. Moreover, let $\xi_u = (\xi_u^{(1)}, \xi_u^{(2)})$ be the increment of the process when leaving $u$ and recall that $m_u = \sum_{z\sim u} \omega_{uz}$ is its vertex weight. By definition of $X^h$, 
		\begin{equation*}
		\mathbb{P}[\xi_u = v-u] = \frac{\omega_{uv}}{m_u}\quad\text{for } v\sim u\ .
		\end{equation*}
		Hence, 
		\begin{equation*}
		\mathbb{E}[\xi_u] = \sum_{v\sim u} (v-u)\, \mathbb{P}[\xi_u = v-u] = \frac{A_u}{m_u} \Delta_h (\operatorname{Id}-u) 
		\ .
		\end{equation*}
		Likewise,
		\begin{equation*}
		\begin{split}
		\mathbb{E}\left[(\xi_u^{(1)})^2\right] &= \frac{A_u}{m_u} \Delta_h (\operatorname{Re}(\operatorname{Id}-u))^2 \ ,  \\
		\mathbb{E}\left[(\xi_u^{(2)})^2\right] &= \frac{A_u}{m_u} \Delta_h (\operatorname{Im}(\operatorname{Id}-u))^2 \ ,  \\
		\mathbb{E}\left[\xi_u^{(1)}\xi_u^{(2)}\right] &= \frac{A_u}{m_u} \Delta_h \left[ (\operatorname{Re}(\operatorname{Id}-u))(\operatorname{Im}(\operatorname{Id}-u))\right],\ 
		\end{split}
		\end{equation*}
		so Lemma \ref{lem:second-order-approx} yields
		\begin{equation}\label{eq:increments}
		\mathbb{E} \left[\xi_u^{(j)}\right] = 0 \ , \qquad \mathbb{E}\mleft[(\xi_u^{(j)})^2\mright] = \frac{2A_u}{m_u} = \frac{1}{\lambda (u)}\quad \text{and}\quad \mathbb{E}\mleft[\xi_u^{(1)} \xi_u^{(2)}\mright] = 0 \ 
		\end{equation}
		for $j\in\{1,2\}$ and $h>0$. The first equality in \eqref{eq:increments} shows that $(M^{\beta, h}_t)_{t\geq 0}$ is a martingale. 
		In order to prove the remaining part of the lemma, we consider the infinitesimal generator of the process $X^{h}$ given by
		\begin{equation*}
		Gf(u) = \frac{1}{2}(\Delta_h f)(u) = \frac{\lambda(u) A_u}{m_u} (\Delta_h f) (u)
		\end{equation*}
		for bounded, measurable functions $f: V_h \rightarrow \mathbb{R}$.
		Define the stopping times $\tau_k := \inf\{ t\geq 0 : \lvert X^h_t \rvert \geq k\}$, $k\in\mathbb{N}$ and denote the stopped process by $X^{h,\tau_k}$.
		Classical theorems on the martingale problem (\cite{ethier1986}, Proposition 1.7, Ch. 4) show that
		\begin{equation*}
		\left(f(X^h_t) - \int_0^t Gf(X^h_s)\,\text{d}s\right)_{t\geq 0}
		\end{equation*}
		is a martingale with respect to the natural filtration of $X^h$ for all bounded functions $f:\mathbb{C}\rightarrow \mathbb{R}$.
		In particular, the process
		\begin{equation*}
		\left(X^{h,\tau_k,(1)}_\cdot\right)^2 - \int_{0}^{\cdot \wedge \tau_k} \frac{\lambda(X_s^h)}{m_{X_s^h}} \sum_{X_s^h + \xi \sim X_s^h} \omega_{X_s^h, \xi} \left((\operatorname{Re}(X_s^h + \xi))^2 - (\operatorname{Re} X_s^h)^2\right)\text{d}s
		\end{equation*}
		is a martingale with respect to the natural filtration of $X^h$, where $\omega_{u, \xi} := \omega_{u(u+\xi)}$ and $t_1 \wedge t_2 := \min\{t_1, t_2\}$. Here, we use that $z\mapsto (\operatorname{Re} z)^2$ is bounded on the compact set $\{z\in \mathbb{C} : \lvert z \rvert \leq k\}$. By the uniqueness of the predictable quadratic variation,
		\begin{equation*}
		\begin{split}
		\langle X^{h,\tau_k, (1)}\rangle_t &= \int_{0}^{t \wedge \tau_k} \frac{\lambda(X_s^h)}{m_{X_s^h}} \sum_{X_s^h + \xi \sim X_s^h} \omega_{X_s^h, \xi} \left((\operatorname{Re}(X_s^h + \xi))^2 - (\operatorname{Re} X_s^h)^2\right)\text{d}s \\
		&= \int_{0}^{t \wedge \tau_k} \frac{\lambda(X_s^h)}{m_{X_s^h}} \sum_{X_s^h + \xi \sim X_s^h} \omega_{X_s^h, \xi} \left((\operatorname{Re}\xi)^2 + 2 (\operatorname{Re} X_s^h) (\operatorname{Re} \xi) \right)\text{d}s \\
		&= \int_0^{t \wedge \tau_k} \lambda (X_s^h)\, \mathbb{E}\mleft[\mleft(\xi_{X_s^h}^{(1)}\mright)^2\mright] \text{d}s = \int_{0}^{t \wedge \tau_k} \lambda(X_s^h) \frac{2 A_{X_s^h}}{m_{X_s^h}}\, \text{d}s = t\wedge \tau_k
		\end{split}
		\end{equation*}
		for all $k\in\mathbb{N}$ and $t\geq 0$, where the last equality follows from \eqref{eq:increments}. Likewise, one can show that  $\langle X^{h,\tau_k, (2)}\rangle_t = t\wedge \tau_k$.
		Since 
		\begin{equation*}
		\lim_{k\rightarrow\infty} \mathbb{P}[\tau_k < t] = 0
		\end{equation*}
		for all $t\geq 0$,
		\begin{equation*}
		\langle X^{h,(j)} \rangle_t = \lim_{k\rightarrow\infty} \langle X^{h,(j)} \rangle_{t\wedge \tau_k} = \lim_{k\rightarrow\infty} \langle X^{h,\tau_k,(j)} \rangle_{t} = \lim_{k\rightarrow\infty} t \wedge \tau_k = t \ ,
		\end{equation*}
		and therefore
		\begin{equation*}
		\langle M^{\beta, h, (j)}\rangle_t = h^{-\beta} \langle X^{h,(j)}_{h^\beta \cdot}\rangle_t = t \ .
		\end{equation*}
		The same arguments  prove $\langle M^{\beta, h, (1)},M^{\beta, h, (2)} \rangle_t = 0$ for all $h>0$ and $t\geq 0$.
	\end{proof}
	Lemma \ref{lem:quad-var} immediately implies a functional central limit theorem for $M^{\beta,h}$, which holds for all $\beta \in [0,2)$:
	\begin{corollary}\label{cor:clt}
		Let $B$ be the Brownian motion on $\mathbb{C}$ with $B_0 = 0$. Then, for $0 \leq \beta < 2$
		\begin{equation*}
		M^{\beta, h} \rightarrow B \ ,
		\end{equation*}
		in distribution as $h\rightarrow 0$.
	\end{corollary}
	Notice, however, that such an CLT is not sufficient in order to prove Theorem \ref{thm:euclideanregime}. 
	\begin{proof}
		We prove the corollary by using the functional central limit theorems for semimartingales obtained in \cite{helland1982}:
		From Lemma \ref{lem:quad-var}, we have $\langle M^{\beta, h, (j)}, M^{\beta, h, (k)} \rangle_t = \delta_{jk}t$ for all $h > 0$. By using that $\lvert e\rvert \leq h$ for all edges $e\in E_h$,
		\begin{equation*}
		\lvert \Delta M^{\beta, h}_t \rvert \leq h^{-\frac{\beta}{2}} \sup_{e\in E_h} \lvert e \rvert \leq h^{\frac{2-\beta}{2}} \to 0 \quad \text{as } h\to 0
		\end{equation*}
		for $t\geq 0$ and $\beta < 2$, where $\Delta M^{\beta,h}_t := M^{\beta, h}_t - M^{\beta,h}_{t-}$. 
		Therefore,
		\begin{equation*}
		\sum_{s\leq t} \left\lvert \Delta M^{\beta, h}_s \right\rvert^2 \mathbbm{1}_{\{\lvert \Delta M^{\beta, h}_s\rvert > \varepsilon \}} \rightarrow 0\quad \text{as } h\to 0
		\end{equation*}
		holds for all $\varepsilon > 0$. 
		Combined with the convergence of the initial vertex $\pi_{V_h} (0) \to 0$ as $h\to 0$, this proves that the assumptions in Theorem 5.4, \cite{helland1982} are satisfied and hence the claimed convergence.
	\end{proof}
	
	\begin{remark}
		We expect that it is also possible to deduce a local central limit theorem by following the strategy of \cite{croydon2008}: 
			The required Hölder continuity of the heat kernel can be obtained from a parabolic Harnack inequality, see \cite{nash1958}, which is equivalent 
			to the volume growth property and Poincaré inequality proven below in Lemma~\ref{lem:poincare}. For similar arguments for variable speed random walks among random conductances, see, e.g., \cite{andres2021,barlow2010,kumagai2014random}.
	\end{remark}
	In contrast to the CLT, the following large deviation principle contains enough information about the tails for proving the desired short-time asymptotics.
	\begin{proposition}\label{prop:ldp}
		Let $\beta \in (0,1)$. The path measures on the space of càdlàg functions induced by $\left(X^h_{h^\beta \cdot}\right)_{h>0}$ with $X_0^{h} \rightarrow x$ as $h\rightarrow 0$ satisfy a large deviation principle with respect to the locally uniform topology, rate function \begin{equation*}
		I(\phi ) = \begin{cases}
		\int_0^\infty \Lambda^* (\partial_t \phi (t)) \text{d}t & \phi \in \mathcal{AC}_x \\
		\infty & \, \text{otherwise} \ ,
		\end{cases}
		\end{equation*}
		and speed $h^{\beta}$, where 
		\begin{equation*}
		\Lambda^* (v) = \sup_{\lambda\in\mathbb{C}}\left( \lambda \overline{v} -\frac{1}{2} \lvert \lambda\rvert^2\right) = \frac{1}{2} \lvert v\rvert^2 
		\end{equation*}
		and $\mathcal{AC}_x$ is the set of all absolutely continuous functions $\phi: \mathbb{R}_{\geq 0}\rightarrow\mathbb{C}$ with $\phi (0) = x$.
	\end{proposition}
	
	
	\begin{proof}
		The proof hinges on Corollary 6.4 in \cite{puhalskii1994}, so we check its conditions $(0)$, $(C_0)$, $(\sup A)$, and $(VS)$. Due to the technicality of this conditions, we omit stating them in detail.
		$(\sup A)$ is only of importance for semimartingales and 
		is thus trivially satisfied in our case since $X^h$ is a martingale.
		The convergence $X_0^h \rightarrow x$ as $h\rightarrow 0$ implies $(0)$.
		By Lemma~\ref{lem:quad-var},
		\begin{equation*}
		h^{-\beta} \left\langle X^{h,(j)}_{h^\beta\cdot},X^{h,(k)}_{h^\beta\cdot}\right\rangle_t = \langle M^{\beta, h, (j)},M^{\beta, h, (k)}\rangle_t = \delta_{jk}t
		\end{equation*}
		for all $t\geq0$ and $j,k\in\{1,2\}$, which implies $(C_0)$ with $C_t = t\operatorname{Id}$. 
		In order to show $(VS)$, let $a>0$ and use the upper bound $\lvert \Delta X_s^h\rvert \leq \sup_{e\in E_h} \lvert e \rvert \leq h$ to get
		\begin{equation*}
		0\leq \sum_{s\leq t} \mathbbm{1}_{\{h^{-\beta} \lvert\Delta X_s^h\rvert > a\}} \leq \sum_{s\leq t} \mathbbm{1}_{\{\Delta X_s^h\neq 0\}} \mathbbm{1}_{\{h^{1-\beta} > a\}}\ .
		\end{equation*}
		for all $t\geq 0$. The latter sum is zero if $h^{1-\beta} \leq a$, which is always satisfied for sufficiently small $h$ since $\beta < 1$. 
		Hence, 
		\begin{equation*}
		\sum_{s\leq t} \mathbbm{1}_{\{h^{-\beta} \lvert\Delta X_s^h\rvert > a\}} = 0
		\end{equation*}
		for all $h\leq a^{\frac{1}{1-\beta}}$.
		This implies $(VS)$. Now, we can apply Corollary 6.4 in \cite{puhalskii1994} in order to obtain the LDP in Skorokhod topology. 
		The desired LDP with respect to the finer locally uniform topology follows since the rate function is infinite for all non-continuous functions, see Theorem~C in \cite{puhalskii1994B}.
	\end{proof}
	
	\begin{corollary}\label{cor:ldp}
		For $\beta \in (0,1)$, $T > 0$ and an open, non-empty set $U\subset\mathbb{C}$,
		\begin{equation*}
		\begin{split}
		\liminf_{\substack{h\rightarrow 0 \\ x' \to x}} h^\beta \log \mathbb{P}_{x'}[X_{h^\beta T}^h \in U] &\geq -\frac{1}{2T} \inf_{u\in U} \lvert u-x\rvert^2 \quad \text{and} \\
		\limsup_{\substack{h\rightarrow 0 \\ x' \to x}} h^\beta \log \mathbb{P}_{x'}[X_{h^\beta T}^h \in \overline{U}] &\leq -\frac{1}{2T} \inf_{u\in \overline{U}} \lvert u-x\rvert^2\ ,
		\end{split}
		\end{equation*}
		where $\overline{U}$ is the closure of $U$.
	\end{corollary}
	\begin{proof}
		The set of paths $M_U = \{\phi: \mathbb{R}_+\to \mathbb{C} \mid \phi \text{ càdlàg},\, \phi(T) \in U\}$ is open in the topology of locally uniform convergence since $U$ is open. Applying the lower bound of the LDP \eqref{eq:ldp-bounds} obtained by Proposition \ref{prop:ldp}, we get
		\begin{equation*}
		\liminf_{\substack{h\rightarrow 0 \\ x' \to x}} h^\beta \log \mathbb{P}_{x'}[X^h_{h^\beta T} \in U] \geq -\inf_{\phi \in M_U} I(\phi)
		\end{equation*}
		with
		\begin{equation*}
		\inf_{\phi \in M_U} I(\phi) = \inf_{\phi \in M_U\cap \mathcal{AC}_x} \frac{1}{2}\int_0^\infty \lvert \partial_t \phi(t) \rvert^2 \text{d}t = \inf_{\phi \in M_U\cap \mathcal{AC}_x} \frac{1}{2}\int_0^T \lvert \partial_t \phi(t) \rvert^2 \text{d}t \ ,
		\end{equation*}
		where the last equality can be justified by choosing a path with $\phi(t) = \phi (T)$ for all $t\geq T$.
		Let $\bar{u} \in \argmin_{u\in \overline{U}} \lvert u-x\rvert$. The minimum of the integral is attained by the line joining $x$ and $\bar{u}$, i.e., $\bar{\phi} (t) = \frac{t}{T}\bar{u} + \left(1-\frac{t}{T}\right)x$. 
		Hence,
		\begin{equation*}
		-\inf_{\phi \in M_U} I(\phi) = -\frac{1}{2T} \lvert\bar{u}-x\rvert^2 = -\frac{1}{2T} \inf_{u\in U} \lvert u-x\rvert^2\ .
		\end{equation*}
		The upper bound is deduced analogously.
	\end{proof}
	
	\subsection{Gaussian lower bound}\label{ssec:densities}
	
	In order to finish the proof of Theorem \ref{thm:euclideanregime}, we need to extend the convergence analysis of $\mathbb{P}[X_{h^\beta t}^h \in A]$ for measurable $A\subset\mathbb{C}$ in the above LDP to the probability density of $\mathbb{P}$, i.e., its heat kernel. To this end, we apply an asymptotic lower bound (Lemma \ref{lem:uniform-lower-bound}), which we obtain by using Gaussian heat kernel estimates (Lemma \ref{lem:gaussian-bound}). The latter is proven by following the usual strategy of establishing a \emph{uniform volume growth} property and a \emph{Poincaré inequality} (Lemma \ref{lem:poincare}). The proofs of all three lemmas are deferred to the end of the this section.
	
	Let $B_r^h (u) = \{v\in V_h : d_h (v,u) \leq r\}$ denote the combinatorial $r$-ball around $u\in V_h$ for any $r\geq 0$ and 
	\begin{equation*}
		\operatorname{vol}_h (u,r) = \sum_{v\in B_r^h(u)} h^{-2} A_v \ .
	\end{equation*}
	
	\begin{lemma}\label{lem:poincare}
		Suppose Assumptions \ref{ass:bap-primal} and \ref{ass:bap-dual} hold. Then there exist constants $c_1, c_2 > 0$ independent of $n$ and $h$ such that
		\begin{equation}\tag{UVG}\label{eq:volume-doubling}
		c_1 n^2 \leq \operatorname{vol}_h(u,n) \leq c_2 n^2
		\end{equation}
		for all $n\in\mathbb{N}$, $u \in V_h$ and $h>0$. Moreover, there is a constant $C_P > 0$ independent of $n$ and $h$ such that
		\begin{equation}\label{eq:poincare}\tag{PI}
		\sum_{v\in B_n^h(u)} h^{-2} A_v ( f(v)-\bar{f}_{B_n^h(u)})^2 \leq C_P\, n^2\sum_{v,w\in B_{2n}^h(u)} \omega_{vw}(f(w)-f(v))^2
		\end{equation}
		for all $f: V_h\rightarrow \mathbb{R}$, $u\in V_h$, $n\in\mathbb{N}$, and $h>0$, where
		\begin{equation*}
		\bar{f}_{B_n^h(u)} = \frac{1}{\operatorname{vol}_h(u,n)} \sum_{v \in B^h_n(u)} h^{-2}A_v f(v)\ .
		\end{equation*}
	\end{lemma}
	
	One can interpret the uniform volume growth property and Poincaré inequality in Lemma~\ref{lem:poincare} as being associated to a weighted graph $(V_h, E_h)$ with edge weights $\omega_{uv}$ and vertex weights $h^{-2} A_u$ for $u,v\in V_h$. The Laplacian of this graph reads $h^2 \Delta_h$, where $\Delta_h$ is defined in Equation \eqref{eq:laplacian}.
	
	Using \eqref{eq:volume-doubling} and \eqref{eq:poincare}, we derive a Gaussian lower bound for the heat kernel $\tilde{p}^h$ of the process generated by $\frac{h^2}{2} \Delta_h$, i.e., the solution of
	\begin{equation*}
	\partial_t \tilde{p}_t^h (u,\cdot) = \frac{h^2}{2} \Delta_h \tilde{p}_t^h (u,\cdot ) 
	\end{equation*}
	for $t>0$ with $\tilde{p}_0^h (u,v) = h^{2} A_u^{-1}$ if $u=v$ and $0$ otherwise. 
	Notice that $\tilde{p}^h$ satisfies 
	\begin{equation*}
	p_{t}^h = h^{-2} \tilde{p}^h_{h^{-2}t}
	\end{equation*}
	for all $t\geq 0$, where $p^h$ is defined in Section \ref{ssec:discrete-hk}.
	
	\begin{lemma}\label{lem:gaussian-bound}
		Suppose Assumptions \ref{ass:bap-primal} and \ref{ass:bap-dual} hold. Then, there exist constants $C_l, c_l > 0$ independent of $h$ such that
		\begin{equation*}
		\tilde{p}_t^h (u,v) \geq \frac{c_l}{\operatorname{vol}_h (u,\sqrt{t})} e^{-C_l d_h (u,v)^2 / t}
		\end{equation*}
		for all $u,v\in V_h$, $h>0$ and $t\geq d_h (u,v)$.
	\end{lemma}
	
	Although it is well known that in many situations a volume growth property and a Poincaré inequality (as in Lemma \ref{lem:poincare}) are equivalent to Gaussian heat kernel estimates, there exist, to the best of our knowledge, no corresponding result for variable speed random walks on graphs with arbitrary vertex measure in the literature. Due to the bounds in Lemma \ref{lem:bounded-weights}, however, an adaptation of the proof in \cite{delmotte1999} to our case is straightforward, see our discussion in the proof of Lemma~\ref{lem:gaussian-bound} below. With the Gaussian lower bound we can prove:
	
	\begin{lemma}\label{lem:uniform-lower-bound}
		Suppose Assumptions \ref{ass:bap-primal} and \ref{ass:bap-dual} hold. 
		For $\beta \in (0,1)$ and $t>0$
		\begin{equation*}
		\liminf_{\substack{h \rightarrow 0 \\ \lvert u-v\rvert \rightarrow 0}} h^\beta \log p_{h^\beta t}^h (u,v) \geq 0 \ .
		\end{equation*}
	\end{lemma}
	
	Before proving Lemma \ref{lem:poincare}, \ref{lem:gaussian-bound} and \ref{lem:uniform-lower-bound}, we complete the proof of Theorem~\ref{thm:euclideanregime} by using Corollary \ref{cor:ldp} and Lemma \ref{lem:uniform-lower-bound}. The arguments are similar to those given in \cite{varadhan1967diffusion}, Section 4. 
	
	\begin{proof}[Proof of Theorem \ref{thm:euclideanregime}] \emph{Lower bound:}
		Let $\varepsilon > 0$. For all $x,y\in\mathbb{C}$, the Chapman--Kolmogorov equation gives
		\begin{equation*}
		\begin{split}
		p_t^h (x,y) &= \sum_{z\in V_h} A_z\, p^h_{(1-\varepsilon)t}(x,z) \, p^h_{\varepsilon t} (z,y) \geq \sum_{z\in V_h \cap B_\delta (y)} A_z\, p_{(1-\varepsilon)t}^h(x,z) \, p_{\varepsilon t}^h (z,y) \\
		&\geq \left(\inf_{z\in B_\delta (y)} p_{\varepsilon t}^h (z,y)\right) \sum_{z\in V_h \cap B_\delta (y)} A_z\, p_{(1-\varepsilon)t}^h (x,z) \ ,
		\end{split}
		\end{equation*}
		where $B_\delta (y)$ denotes the Euclidean $\delta$-ball around $y$ for any $\delta > 0$.
		By using Corollary \ref{cor:ldp}, the bound in Lemma \ref{lem:uniform-lower-bound}, and taking $\delta\rightarrow 0$, we get
		\begin{equation*}
		\begin{split}
		\liminf_{\substack{h\rightarrow 0 \\ x' \rightarrow x \\ y'\rightarrow y}}&\ h^\beta \log p_{h^\beta t}^h (x',y') \\ &\geq \liminf_{\substack{h\rightarrow 0 \\ x' \rightarrow x \\ y'\rightarrow y}} h^\beta \left(\inf_{z\in B_\delta (y)} \log p_{\varepsilon h^\beta t}^h (z,y') + \log \mathbb{P}_{x'}\left[X^h_{(1-\varepsilon)h^\beta t} \in B_\delta (y)\right] \right) \\
		&\geq -\frac{1}{2t (1-\varepsilon)} \lvert x-y\rvert^2\ .
		\end{split}
		\end{equation*}
		Since $\varepsilon >0$ is arbitrary, the lower bound holds true. \\
		\emph{Upper bound:} For $\delta > 0$ and $y'\in B_\delta (y)$
		\begin{equation*}
		p_{h^\beta t}^h (x,y') \leq \sum_{z\in \overline{B}_\delta (y)} p_{h^\beta t}^h (x,z)\ . 
		\end{equation*}
		By Corollary \ref{cor:ldp},
		\begin{equation*}
		\begin{split}
		\limsup_{\substack{h \rightarrow 0 \\ x'\rightarrow x \\ y'\rightarrow y}} h^\beta &\log p_{h^\beta t}^h (x',y') \\
		&\leq \limsup_{\substack{h \rightarrow 0 \\ x' \rightarrow x \\ y'\rightarrow y}} h^\beta\, \log \mathbb{P}_{x'}\left[X^h_{h^\beta t} \in \overline{B}_\delta (y)\right] \leq -\frac{1}{2t} \inf_{z\in \overline{B}_{\delta}(y)} \lvert z-x\rvert^2\ .
		\end{split}
		\end{equation*}
		We obtain the desired upper bound by letting $\delta \rightarrow 0$.
	\end{proof}
	
	Now, we prove Lemma \ref{lem:poincare}, \ref{lem:gaussian-bound} and \ref{lem:uniform-lower-bound}.
	
	\begin{proof}[Proof of Lemma \ref{lem:poincare}]
		By the definition of $A_u$ and Lemma \ref{lem:spanning-tree},
		\begin{equation*}
		\sum_{v\in B_n^h (u)} A_v \leq \pi \max_{v\,:\, d_h (u,v) = n+1} \lvert v-u\rvert^2 \leq \pi h^2 (n+1)^2 \leq 4 \pi h^2 n^2
		\end{equation*}
		for all $n\in\mathbb{N}$ because the Euclidean ball of radius $\max_{v\,:\, d_h (u,v) = n+1} \lvert v-u\rvert$ centered in $u$ covers the union of all dual faces corresponding to vertices in the combinatorial $n$-ball $B_n^h (u)$. 
		
		In order to obtain the lower bound notice that $d_h (u,v) \geq \frac{\kappa}{c_p}$ for $u,v\in V_h$ implies by Lemma \ref{lem:spanning-tree} that $\lvert v-u \rvert\geq h$. Fix $n\in\mathbb{N}$ such that $n\geq \frac{\kappa}{c_p}$. Then, the Euclidean ball with radius $$\min_{v\,:\, d_h (u,v) = n} \left(\lvert v-u \rvert - h\right)$$ centered in $u$ is contained in the union of dual faces corresponding to vertices in $B_n^h (u)$. Hence, we obtain by Lemma \ref{lem:spanning-tree} that
		\begin{equation*}
		\sum_{v\in B_n^h (u)} A_v \geq \pi \min_{v\,:\, d_h (u,v) = n} \left(\lvert v-u \rvert - h\right)^2 \geq \pi h^2 \left(\frac{c_p}{\kappa}  n - 1\right)^2,
		\end{equation*}
		which implies the existence of a constant $\bar{c}>0$ satisfying
		\begin{equation*}
		\sum_{v\in B_n^h (u)} A_v \geq \bar{c} h^2 n^2 \ .
		\end{equation*}
		By possibly adjusting the constant $\bar{c}$ further for finitely many $n$ 
		and using the uniform lower bound on $A_v$ in Lemma \ref{lem:bounded-weights} we get such a uniform lower bound for all  $n\in\mathbb{N}$. This proves \eqref{eq:volume-doubling}.
		
		In order to prove the Poincaré inequality \eqref{eq:poincare}, we follow the approach of Theorem 3.4 in \cite{angel2016} based on the method of canonical paths originally developed in \cite{diaconis1991,jerrum1989}: 
		Fix $n\in\mathbb{N}$ and a function $f : V_h \rightarrow \mathbb{R}$. Let $Z$ and $Z'$ be two i.i.d. $B_n^h (u)$-valued random variables with distribution 
		$$\mathbb{P}[Z=v] = \frac{h^{-2}A_v}{\operatorname{vol}_h (u,n)}$$
		for all $v\in B_n^h(u)$. 
		Then, $\bar{f}_{B_n^h(u)} = \mathbb{E}[f(Z)]$ and we get
		\begin{equation}\label{eq:poinc-exp}
		\begin{split}
		\sum_{v\in B_n^h(u)} h^{-2} A_v\,(f(v)-&\bar{f}_{B_n^h(u)})^2 = \operatorname{vol}_h (u,n)\, \mathbb{E} \left[ (f(Z) - \mathbb{E}[ f(Z)])^2  \right]\\
		&= \operatorname{vol}_h (u,n)\left( \mathbb{E}\left[f^2 (Z)\right] - \mathbb{E}[f(Z)]^2\right)\\
		&= \frac{\operatorname{vol}_h (u,n)}{2}\mathbb{E}\left[ f^2(Z) + f^2(Z') -2f(Z)f(Z')\right]\\
		&=\frac{\operatorname{vol}_h (u,n)}{2}\mathbb{E}\left[ (f(Z)-f(Z'))^2\right]\ .
		\end{split}
		\end{equation}
		For any vertex $v\in V_h$, let $\widehat{v}$ be an independent random variable that is uniformly distributed in the corresponding dual face $D_v$ of $v$, i.e.,
		\begin{equation*}
		\mathbb{P}\left[\widehat{v}\in F\right] = \frac{\lambda(F\cap D_v)}{A_v}
		\end{equation*}
		for all measurable sets $F\subset\mathbb{C}$, where $\lambda$ denotes the Lebesgue measure on $\mathbb{C}$.
		We choose a (possibly random) path $[v,w]\in \mathfrak{X}(v,w)$ for all combinations of two vertices $v,w\in B_n^h (u)$ 
		consisting only of vertices adjacent to faces crossed by the line segment connecting $\widehat{v}$ and $\widehat{w}$.
		The identity
		\begin{equation*}
		f(v)-f(w) = \sum_{\{z_i,z_{i+1}\}\in [v,w]} f(z_{i+1})-f(z_i)
		\end{equation*}
		and the Cauchy--Schwarz inequality show
		\begin{equation}\label{eq:path-telescoping}
		\left(f(v)-f(w) \right)^2 \leq n( [v,w] ) \sum_{\{z_i,z_{i+1}\}\in [v,w]} \left( f(z_{i+1})-f(z_i) \right)^2 \ ,
		\end{equation}
		where as before $n(\gamma )$ denotes the combinatorial length of a path $\gamma$.
		By using Theorem 1 in \cite{xia2013}, i.e., the same argument as in the proof of Lemma \ref{lem:spanning-tree}, and Assumption \ref{ass:bap-primal}, we obtain
		\begin{equation*}
		n( [v,w] ) \leq \frac{\kappa}{c_p h} \lvert w-v\rvert \ ,
		\end{equation*}
		where $\kappa$ is the same constant as in Lemma \ref{lem:spanning-tree}. Then, Lemma \ref{lem:spanning-tree} and $v,w\in B_n^h (u)$ yield
		\begin{equation*}
		n( [v,w] ) \leq \frac{2 \kappa}{c_p} n \ .
		\end{equation*}
		Combining this with Equation \eqref{eq:poinc-exp}, the bound \eqref{eq:path-telescoping}, $\operatorname{vol}_h(u,n)\leq c_2n^2$ by \eqref{eq:volume-doubling}, and Lemma \ref{lem:bounded-weights},
		we get
		\begin{equation*}
		\begin{split}
		\sum_{v\in B_n^h(u)} h^{-2} A_v &\left( f(v)-\bar{f}_{B_n^h(u)}\right)^2 \leq C_1 n^3\,\mathbb{E}\left[ \sum_{\{z_i,z_{i+1}\}\in [Z,Z']} \left( f(z_{i+1})-f(z_i) \right)^2\right] \\
		&\leq C_2 n^3\,\mathbb{E}\left[ \sum_{v,w\in B_{cn}^h (u)} 
		\omega_{vw} \left( f(w)-f(v) \right)^2\, \mathbbm{1}_{\{w\in [Z,Z']\}}\right] \\
		&= C_2 n^3 \sum_{v,w\in B_{cn}^h (u)} \omega_{vw} \left( f(w)-f(v) \right)^2\, \mathbb{P}[w\in [Z,Z']]
		\end{split}
		\end{equation*}
		with constants $C_1,C_2>0$ and $c\geq 1$ such that $[v,w]$ is in $B_{cn}^h(u)$ for all $v,w\in B_n^h (u)$. Without loss of generality, we can choose $c=2$ as the sum on the left only depends on the values of $f$ on $B_n^h (u)$. 
		By using similar arguments as in \cite{angel2016},~Theorem 3.4 we will now show that there is a constant $C_3 > 0$ with
		\begin{equation}\label{eq:line-crossing}
		\mathbb{P}[w\in [Z,Z']] \leq \frac{C_3}{n}
		\end{equation}
		for all $w\in B_{2n}^h (u)$. From \eqref{eq:line-crossing} we can conclude
		\begin{equation*}
		\begin{split}
		\sum_{v\in B_n^h(u)} h^{-2} A_v \left( f(v)-\bar{f}_{B_n^h(u)}\right)^2 \leq C_P\, n^2 \sum_{v,w\in B_{2n}^h(u)} \omega_{vw} (f(w)-f(v))^2
		\ ,
		\end{split}
		\end{equation*}
		which is \eqref{eq:poincare}.
		
		In order to prove inequality \eqref{eq:line-crossing}, let $\widehat{Z}$ and $\widehat{Z}'$ be independent random variables uniformly distributed in the dual faces of the random vertices $Z$ and $Z'$, respectively, i.e.,
		\begin{equation*}
		\mathbb{P}\left[\widehat{Z}\in F\right] = \sum_{v\in B_{n}^h(u)} \mathbb{P}\left[\widehat{v}\in F\mid Z = v\right]\, \mathbb{P} [Z = v] = \sum_{v\in B_{n}^h(u)}\frac{\lambda (F\cap D_v)}{A_v} \frac{h^{-2} A_v}{\operatorname{vol}_h (u,n)} 
		\end{equation*}
		for measurable sets $F\subset\mathbb{C}$. 
		By applying
		the volume growth property \eqref{eq:volume-doubling} once again we get
		\begin{equation*}
		\mathbb{P}\left[\widehat{Z}\in F\right] \leq \frac{C_4}{n^2 h^2} \lambda (F)
		\end{equation*}
		and thus the following bound on the density of $\widehat{Z}$ with respect to $\lambda$, which exists by the Radon--Nikodym theorem:
		\begin{equation*}
		\frac{\text{d}\widehat{Z}}{\text{d}\lambda} \leq \frac{C_4}{n^2 h^2} \ .
		\end{equation*}
		Let $A$ be the event that the line segment connecting $\widehat{Z}$ and $\widehat{Z}'$ intersects with $B_h(w)$. 
		By Lemma \ref{lem:spanning-tree}, there is a constant $C_5$ such that the Euclidean ball $B_{C_5nh}(w)$ of radius $C_5nh$ around $w$ satisfies $$\operatorname{supp} \widehat{Z} = \bigcup_{v \in B_{n}^h (u)} D_v \subset B_{C_5nh}(w)$$
		for all $w\in B^h_{2n} (u)$.
		We obtain
		\begin{equation*}
		\begin{split}
		\mathbb{P}[w\in [Z,Z']] 
		= \int_{\mathbb{C}} \int_{\mathbb{C}} \mathbbm{1}_A\, \text{d}\widehat{Z} \text{d}\widehat{Z}' \leq \int_{B_{C_5nh}(w)} \int_{B_{C_5nh}(w)} \mathbbm{1}_A \cdot \left(\frac{C_4}{n^2 h^2}\right)^2 \text{d}x\text{d}y\ .
		\end{split}
		\end{equation*}
		A scaling argument shows that the expression on the right hand side is proportional to 
		the probability that the line segment joining two uniformly distributed points in the ball of radius $C_5$ centered in $0$ intersects with $B_{n^{-1}}(0)$. But the minimal distance between such a random line segment and the origin is a continuous random variable with finite density such that the probability of intersecting  $B_{n^{-1}}(0)$ is bounded by a constant times $n^{-1}$. This concludes the proof of \eqref{eq:line-crossing}.
	\end{proof}
	
	\begin{proof}[Proof of Lemma \ref{lem:gaussian-bound}]
		The proof of the Gaussian lower bound follows from a careful line-by-line modification of the respective implication in \cite{delmotte1999}: Since Lemma \ref{lem:bounded-weights} shows that the vertex weights $h^{-2} A_u$ are bounded from above and away from zero with constants independent of $h$, the weighted Poincaré and Sobolev--Poincaré inequalities can be deduced by using Lemma \ref{lem:poincare}. Moser's iterative technique yields the parabolic Harnack inequality, which implies on-diagonal heat kernel estimates. Then, a classical chaining argument proves Lemma \ref{lem:gaussian-bound}.
	\end{proof}
	
	\begin{remark}
		Alternatively, we can prove Lemma~\ref{lem:gaussian-bound} by using the equivalence between the Poincaré inequality with uniform volume growth property and Gaussian heat kernel estimates in the case of \emph{constant speed random walks} as established in \cite{delmotte1999}: 
		
		With the same methods as in the proof of Lemma \ref{lem:poincare} (but replacing the vertex measure $h^{-2} A_u$ by $m_u$, where $m_u$ is given by Equation \eqref{eq:alternative-weights}) and Theorem 1.7 in \cite{delmotte1999}, we obtain Gaussian estimates for the heat kernel $q^h$ of the operator
		\begin{equation*}
		\widetilde{\Delta}_h f (u) = \frac{1}{m_u} \sum_{v\sim u} \omega_{uv} (f(v)-f(u)) \ .
		\end{equation*}
		The time-continuous random walk $Y_h$ with generator $\widetilde{\Delta}_h$ is a random time-change of the process $X^h$ generated by $\frac12 \Delta_h$. It solves (see Theorem 1.1 in \cite{ethier1986}, Ch. 6)
		\begin{equation*}
		Y^h_t = X^h \left(\int_0^t \frac{2A_{Y^h_s}}{m_{Y^h_s}}\text{d}s\right) \ ,
		\end{equation*}
		where $X^h (s) := X^h_s$. 
		We denote the corresponding random time-change by
		\begin{equation*}
		\tau (t) = \int_0^t \frac{2A_{Y^h_s}}{m_{Y^h_s}}\text{d}s \ .
		\end{equation*}
		Using Lemma \ref{lem:bounded-weights}, we bound
		\begin{equation*}
		\eta^{-1}h^2 \leq \frac{2A_u}{m_u} \leq \eta h^2 \ ,
		\end{equation*}
		for some $\eta > 0$ and all $u\in V_h$, which implies 
		\begin{equation*}
		\eta^{-1}h^{-2} t \leq \tau^{-1} (t) \leq \eta h^{-2} t
		\end{equation*}
		for all $t>0$.
		In particular,
		\begin{equation*}
		\begin{split}
		\mathbb{P}_u \left[X^h_{h^{2}t} = v\right] &= \mathbb{P}_u \left[Y^h_{\tau^{-1} (h^{2}t)} = v\right] \\
		&\geq \inf_{s\in [\eta^{-1} t, \eta t]} \mathbb{P}_u \left[Y^h_{s} = v\right] = m_v \inf_{s\in [\eta^{-1} t, \eta t]} q_s (u,v) ´\ .
		\end{split}
		\end{equation*}
		Applying the Gaussian lower bound on $q^h$ proves Lemma \ref{lem:gaussian-bound}.
	\end{remark}
	
	\begin{proof}[Proof of Lemma \ref{lem:uniform-lower-bound}]
		Fix $u,v\in \mathbb{C}$ and $t > 0$. With Lemma \ref{lem:spanning-tree}, Lemma \ref{lem:gaussian-bound} and a constant $\tilde{C}_l>0$, we obtain
		\begin{equation*}
			\begin{split}
			p_{h^\beta t}^h (u,v) = h^{-2} \bar{p}^h_{h^{\beta - 2} t} (u,v) &\geq \frac{h^{-2} c_l}{\operatorname{vol}_h(u,\sqrt{h^{\beta - 2}t})} e^{-C_l \frac{d_h (u,v)^2}{h^{\beta - 2} t}} \\
			&\geq \frac{h^{-2} c_l}{\operatorname{vol}_h(u,\sqrt{h^{\beta - 2}t})} e^{-\tilde{C}_l \frac{\lvert u-v\rvert^2}{h^\beta t}} \ ,
			\end{split}
		\end{equation*}
		if $d_h (u,v) \leq h^{\beta - 2}t$, which holds for sufficiently small $h$ since $\beta < 1$. This inequality proves the lemma since
		\begin{equation*}
		\begin{split}
		\liminf_{\substack{h \rightarrow 0 \\ \lvert u-v\rvert \rightarrow 0}} h^\beta \log p_{h^\beta t}^h (u,v) &\geq \liminf_{\substack{h \rightarrow 0 \\ \lvert u-v\rvert \rightarrow 0}} \left( - h^\beta \log \operatorname{vol}_h\left(u,\sqrt{h^{\beta - 2}t}\right) - \tilde{C}_l \frac{\lvert u-v\rvert^2}{t}\right) \\
		&\geq \liminf_{h \rightarrow 0} \left( - h^\beta \log \left( h^{\beta -2} t\right)\right) = 0 \ ,
		\end{split}
		\end{equation*}
		where we used the volume growth property \eqref{eq:volume-doubling} in the last inequality.
	\end{proof}
	
	\section{Graph Regime}
	This section is devoted to proving Theorem \ref{thm:graphregime}, which describes the asymptotic behavior of $p^h_{h^\beta t}$ as $h\to 0$ for $\beta > 1$.
	As before, let $\mathfrak{X}_h(u,v)$, $u,v\in V_h$ be the set of paths $\gamma = (q_0,\dots, q_{n(\gamma)})$ in $\Gamma_h$ with $q_0 = u$, $q_{n(\gamma)} = v$ and $q_i \sim q_{i+1}$ for $0\leq i < n(\gamma)$.
	To ease notation, we write 
	$$\mathfrak{X}_h (x,y) := \mathfrak{X}_h (\pi_{V_h}(x), \pi_{V_h} (y))\quad \text{and}\quad d_h (x,y) := d_h (\pi_{V_h}(x), \pi_{V_h} (y))$$
	for all $x,y\in\mathbb{C}$, where $\pi_{V_h}$ denotes the projection to the closest vertex, see Section \ref{ssec:discrete-hk}.
	
	\begin{theorem}[Graph regime]\label{thm:graphregime}
		Let $(\Gamma_h)_{h>0}$ be a family of isoradial graphs and suppose that Assumptions \ref{ass:bap-primal} and \ref{ass:bap-dual} hold. Fix $x,y\in\mathbb{C}$.
		Then,
		\begin{equation*}
		\lim_{h\rightarrow 0} \left[ \frac{h}{\log h^{\beta-1}} \log p^{h}_{h^\beta t} (x,y) - hd_h (x,y)\right] = 0 \ .
		\end{equation*}
		holds for $\beta > 1$ and $t>0$.
	\end{theorem}
	
	\begin{remark}
		Theorem \ref{thm:graphregime} holds true for any family of weighted graphs with uniformly bounded degree satisfying the properties in Lemma \ref{lem:spanning-tree} and Lemma \ref{lem:bounded-weights} (with possibly different constants). The proof does neither require isoradiality nor any specific properties of $\Delta_h$.
	\end{remark}
	
	\begin{example} 
		Let $\Gamma_h = h\mathbb{Z}^2$ for all $h>0$. In this case,
		\begin{equation*}
		\lim_{h\rightarrow 0} hd_h (x,y) = \lVert x-y\rVert_1\ ,
		\end{equation*}
		so Theorem \ref{thm:graphregime} yields
		\begin{equation*}
		\lim_{h\rightarrow 0} \frac{h}{\log h^{\beta-1}} \log p^{h}_{h^\beta t} (x,y) = \lVert x-y\rVert_1
		\end{equation*}
		for all $x,y\in\mathbb{C}$ and $\beta > 1$.
	\end{example}
	
	Recall from Lemma~\ref{lem:bounded-weights} that the Assumptions \ref{ass:bap-primal} and~\ref{ass:bap-dual} imply uniform bounds
	\begin{equation*}
	c_d \leq \inf_{e\in E_h} \omega_{e} \leq \sup_{e\in E_h} \omega_{e} \leq c_p^{-1}
	\end{equation*}
	on the edge weights, uniform bounds
	\begin{equation*}
	\kappa_1 h^{2} \leq \inf_{u\in V_h} A_u \leq \sup_{u\in V_h} A_u \leq \kappa_2 h^2\ .
	\end{equation*}
	on the dual areas, and a uniform bound $M$ on the maximal degree of $\Gamma_h$.
	In particular, there exists a constant $\alpha >0$ such that 
	$$\alpha^{-1} \leq h^2 \mu_{uv} \leq \alpha$$
	for all $u\sim v$ with $\mu_{uv} = \frac{\omega_{uv}}{A_u}$.
	In order to prove Theorem \ref{thm:graphregime}, we will use the following bounds derived in \cite{metzger2000}.
	
	\begin{lemma}[\cite{metzger2000}, Theorem 1]\label{lem:heat-bounds}
		Let $\mathfrak{X}_h(x,y)$ denote the set of paths from $x$ to $y$ in $\Gamma_h$.
		For all $x,y\in \mathbb{C}$ with $x\neq y$, $t\geq 0$ and $h>0$,
		\begin{equation*}
		p_t^h (x,y) \geq \frac{e^{-\alpha h^{-2} M t}}{\sqrt{2\pi}} \sup_{\gamma\in \mathfrak{X}_h(x,y)}\left(\prod_{e\in \gamma}\frac{t\mu_{e}}{\lvert \gamma \rvert}\right)
		\end{equation*}
		and
		\begin{equation*}
		p_t^h (x,y) \leq e^{\alpha h^{-2} Mt+1} \left(\sup_{\gamma\in\mathfrak{X}_h(x,y)} \prod_{e\in\gamma} \frac{h^2 \mu_{e}}{\alpha }\right) \left(\frac{eM\alpha t}{h^2 d_h (x,y)}\right)^{d_h(x,y)} \ .
		\end{equation*}
	\end{lemma}

	For $x\neq y$ and $t\geq 0$, the bounds on the edge weights and the lower bound in Lemma \ref{lem:heat-bounds} imply
	\begin{equation}\label{eq:lower-bound-improved}
	\begin{split}
	p_{h^\beta t}^h (x,y) &\geq \frac{e^{-\alpha h^{\beta-2} M t}}{\sqrt{2\pi}} \sup_{k\geq d_h (x,y)} \left(\frac{\alpha^{-1} h^{\beta-2} t}{k}\right)^{k} \\
	&\geq \frac{e^{-\alpha h^{\beta-2} M t}}{\sqrt{2\pi}} \left(\frac{\alpha^{-1} h^{\beta-2}t}{d_h (x,y)}\right)^{d_h (x,y)}
	\end{split}
	\end{equation}
	for $\beta > 1$ and sufficiently small $h>0$ since then $d_h (x,y) > \alpha^{-1} h^{\beta -2}t$ and for any $c>0$ the function $\mathbb{R}_{>0}\rightarrow\mathbb{R}: x\mapsto \left(\frac{c}{x}\right)^x$ is monotonically decreasing on~$[c, \infty)$. 
	
	\begin{proof}[Proof of Theorem \ref{thm:graphregime}]
		By using the monotonicity of the logarithm and Lemma \ref{lem:heat-bounds}, we estimate
		\begin{align*}
		\log\,&p_{h^\beta t}^h (x,y) \\
		&\leq \alpha Mh^{\beta-2}t +1 + \log \left(\sup_{\gamma\in\mathfrak{X}_h(x,y)} \prod_{e\in\gamma} \frac{h^2 \mu_{e}}{\alpha}\right) + d_h(x,y) \log \left( \frac{eM\alpha h^{\beta-2}t}{d_h (x,y)}\right) \\
		&\leq \alpha Mh^{\beta-2}t +1+d_h(x,y) \left[ \log (eM\alpha h^{\beta-2}t ) - \log d_h (x,y)\right]\ ,
		\end{align*}
		where the last inequality follows by using that $\mu_e \leq \alpha$ for all $e\in E_h$.
		Therefore, with 
		\begin{equation*}
		\lim_{h\rightarrow 0} \frac{hd_h(x,y)}{\log h^{\beta-2}} = 0,
		\end{equation*}
		which holds by Lemma \ref{lem:spanning-tree},
		we obtain
		\begin{equation*}
		\begin{split}
		\lim_{h\rightarrow 0}\bigg[\frac{h}{\log h^{\beta -1}} & \log p_{h^\beta t}^h (x,y)-hd_h (x,y)\bigg] \\
		\leq &\lim_{h\rightarrow 0} \left[hd_h(x,y) \left( \frac{\log (h^{\beta-2} ) - \log d_h (x,y)}{\log h^{\beta -1}} - 1\right)\right] \ .
		\end{split}
		\end{equation*}
		Since 
		\begin{equation*}
		\lim_{h\rightarrow 0}-\frac{\log d_h (x,y) }{\log h} = 1
		\end{equation*}
		by Lemma \ref{lem:spanning-tree}, we get
		\begin{equation*}
		\lim_{h\rightarrow 0} \frac{\log (h^{\beta-2} ) - \log d_h (x,y)}{\log h^{\beta -1}} = \frac{\beta -2}{\beta -1} + \frac{1}{\beta -1} = 1\ ,
		\end{equation*}
		which -- combined with the fact that $hd_h (x,y)$ is bounded in $h$ -- proves
		\begin{equation*}
		\limsup_{h\rightarrow 0}\bigg[\frac{h}{\log h^{\beta -1}} \log p_{h^\beta t}^h (x,y)-hd_h (x,y)\bigg] \leq 0 \ .
		\end{equation*}
		Similarly, the lower bound \eqref{eq:lower-bound-improved} yields
		\begin{equation*}
		\log p_{h^\beta t}^h (x,y) \geq -\alpha M h^{\beta-2} - \log \sqrt{2\pi} + d_h (x,y) \left[\log (\alpha^{-1} h^{\beta-2}) - \log d_h (x,y)\right] \ .
		\end{equation*}
		With the same arguments as above,
		\begin{equation*}
		\liminf_{h\rightarrow 0}\left[\frac{h}{\log h^{\beta -1}} \log p_{h^\beta t}^h (x,y)-hd_h (x,y)\right] \geq 0 \ .
		\end{equation*}
		This concludes the proof.
	\end{proof}
	
	\printbibliography
\end{document}